\documentclass[sn-mathphys,Numbered]{sn-jnl}%

\usepackage{graphicx}%
\usepackage{multirow}%
\usepackage{amsmath,amssymb,amsfonts}%
\usepackage{amsthm}%
\usepackage{mathrsfs}%
\usepackage[title]{appendix}%
\usepackage{xcolor}%
\usepackage{textcomp}%
\usepackage{manyfoot}%
\usepackage{booktabs}%
\usepackage{algorithm}%
\usepackage{algorithmic}%

\usepackage{listings}%
\newtheorem{theorem}{Theorem}

\usepackage{hyperref}       

\newtheorem{lemma}{Lemma}%

\newtheorem{definition}{Definition}%

\newcommand{\sm}[2]{\begin{smallmatrix}\item1\\\item2 \end{smallmatrix}}

\newcommand\ddfrac[2]{\frac{\displaystyle \item1}{\displaystyle \item2}}

\newcommand{\mc}[1]{\mathbb{\item1}}

\usepackage{mathtools}

\usepackage{tikz}
\usepackage{ifthen}
\usetikzlibrary{arrows,automata,positioning,shapes,calc, decorations.pathmorphing}

\newcommand{\floor}[1]{\left\lfloor{#1}\right\rfloor}
\newcommand{\ceil}[1]{\left\lceil{#1}\right\rceil}

\newcommand{\eps}{\varepsilon}

\newcommand{\one}{\mathbf{1}}

\DeclareMathOperator{\spn}{span}

\DeclareMathOperator{\image}{Im}

\DeclareMathOperator{\diag}{diag}

\newcommand{\N}{\mathbb{N}}
\newcommand{\R}{\mathbb{R}}

\newcommand{\bA}{{\bf A}}
\newcommand{\bB}{{\bf B}}

\newcommand{\bI}{{\bf I}}

\newcommand{\bP}{{\bf P}}

\newcommand{\bW}{{\bf W}}

\newcommand{\cL}{{\mathcal{L}}}
\newcommand{\cLp}{\cL^\perp}
\newcommand{\cM}{{\mathcal{M}}}

\newcommand{\bb}{{\bf b}}

\newcommand{\bg}{{\bf g}}

\newcommand{\bm}{{\bf m}}

\newcommand{\bp}{{\bf p}}
\newcommand{\bq}{{\bf q}}

\newcommand{\bs}{{\bf s}}

\newcommand{\bx}{{\bf x}}

\newcommand{\bz}{{\bf z}}

\newcommand{\smax}{\sigma_{\max}}
\newcommand{\smin}{\sigma_{\min}}
\newcommand{\sminp}{\sigma_{\min}^+}

\newcommand{\lmax}{\lambda_{\max}}
\newcommand{\lmin}{\lambda_{\min}}
\newcommand{\lminp}{\lambda_{\min}^+}

\newcommand{\defeq}{\vcentcolon=}
\newcommand{\eqdef}{=\vcentcolon}

\newcommand{\mA}{\mathbf{A}}
\newcommand{\mB}{\mathbf{B}}

\newcommand{\mW}{\mathbf{W}}

\newcommand{\mP}{\mathbf{P}}

\newcommand{\g}{\nabla}

\newcommand{\mAT}{{\mA^\top}}
\newcommand{\mBT}{{\mB^\top}}

\newcommand{\chiA}{{\chi_A}}

\newcommand{\norm}[1]{\left\| #1 \right\|}

\newcommand{\sqn}[1]{\norm{#1}^2}
\newcommand{\angles}[1]{\left\langle #1 \right\rangle}
\newcommand{\cbraces}[1]{\left( #1 \right)}
\newcommand{\sbraces}[1]{\left[ #1 \right]}

\def\<#1,#2>{\langle #1,#2\rangle}

\newcommand{\muH}{\mu_{H}}
\newcommand{\LH}{L_{H}}

\usepackage{color}

\newcommand{\squeeze}{}  
\begin{document}
\title{Decentralized optimization with affine constraints over time-varying networks}
%
%


\author[1,3]{\fnm{Demyan} \sur{Yarmoshik}}\email{yarmoshik.dv@phystech.edu}
\author[1]{\fnm{Alexander} \sur{Rogozin}}\email{aleksandr.rogozin@phystech.edu}
\author[1,2,3]{\fnm{Alexander} \sur{Gasnikov}}\email{gasnikov@yandex.ru}

%
\affil[1]{\orgname{Moscow Institute of Physics and Technology}, \orgaddress{\city{Dolgoprudny}, \country{Russia}}}
\affil[2]{\orgname{Skoltech}, \orgaddress{\city{Moscow}, \country{Russia}}}
\affil[3]{\orgname{Institute for Information Transmission Problems}, \orgaddress{\city{Moscow}, \country{Russia}}}

%
\maketitle              

\begin{abstract}

The decentralized optimization paradigm assumes that each term of a finite-sum objective is privately stored by the corresponding agent. 
Agents are only allowed to communicate with their neighbors in the communication graph. 
We consider the case when the agents additionally have local affine constraints and the communication graph can change over time.
We provide the first linearly convergent decentralized algorithm for time-varying networks by generalizing the optimal decentralized algorithm ADOM to the case of affine constraints.
We show that its rate of convergence is optimal for first-order methods by providing the lower bounds for the number of communications and oracle calls.

\end{abstract}

\section{Introduction}
Decentralized optimization is a popular approach for solving modern machine learning and control problems.
For example, training a large language model is a high-dimensional optimization problem with complex objective function which is best solved on a cluster of computational units in a decentralized manner \cite{huang2019gpipe, lian2017can}.
In control of distributed systems, such as drone swarms \cite{hu2021vgai, zhu2023swarm} and wireless sensor networks \cite{li2023optimal}, decentralization is desired because only agent-to-agent communication is possible.
Data privacy and robustness properties of decentralized algorithms make them popular in power systems control \cite{molzahn2017survey, silva2023privacy, wang2022distributed}.
It is common for these problems to have interconnections between variables, usually posed as affine constraints, e.g., direct-current~(DC) power flow constraints in control problems related to electrical energy systems. 
It is also often the case that the communication graph between computation nodes is subject to change during the optimization process.

Over the past decade, constrained distributed optimization has attracted the attention of researchers.
Among the first applications of first-order methods to constrained decentralized optimization was the \textit{projected subgradient algorithm} in \cite{nedic2010constrained}, where the time-varying case was also analyzed. 
A systematic review of main problem classes falling into the definition of distributed constrained optimization along with algorithms working on various levels of decentralization was given in \cite{necoara2011parallel}.
More recent works use first-order methods to deal with a broad range of problem variants, including the
nonconvex objectives \cite{scutari2016parallel,scutari2019distributed},
the composite objectives \cite{wang2022distributed,wu2022distributed},
the inequality constraints \cite{zhu2011distributed,wang2022distributed,wu2022distributed,scutari2016parallel,gong4109852push,liang2019distributed} 
and other assumptions on problem's structure \cite{wang2022distributed,alghunaim2018dual}.
The ADMM-based approaches are also popular \cite{carli2019distributed,aybat2019distributed,chang2016proximal}.

However, to the best of our knowledge, no decentralized linear convergent first-order algorithms for affine-constrained problems have been proposed. 
In this work, we close this gap by providing a linearly convergent dual algorithm for decentralized affine-constrained optimization of the sum of smooth strongly convex functions over time-varying networks.
We build on the recently developed optimal algorithms for decentralized optimization over time-varying networks \cite{kovalev2021adom,kovalev2021lower}, 
and extend these results to the affine-constrained case.
This paper could also be seen as a generalization of \cite{rogozin2022decentralized} to the time-varying networks.

We also show that our new algorithm inherits the optimality of ADOM by constructing lower bounds on the number of communications and oracle calls in time-varying case.
During this analysis we also prove lower bounds for the static communication graph setup, thus showing the optimality of algorithms in \cite{rogozin2022decentralized}.
 
\subsection{Basic definitions and assumptions}
\begin{itemize}
\item Differentiable function $f$ is \textit{$\mu$-strongly convex} if 
\begin{align}
f(y) \geq f(x) + \<\nabla f(x), y -x> + \frac\mu2 \sqn{y-x}~\forall x,y \in \R.
\end{align}
\item Differentiable function $f$ is \textit{$L$-smooth} (or has $L$-Lipschitz continuous gradient) if  
\begin{align}
f(y) \leq f(x) + \<\nabla f(x), y -x> + \frac L2 \sqn{y-x}~\forall x,y \in \R.
\end{align}
\item $\lmin(A)$, $\lminp(A),~\lmax(A)$ are the minimal, the minimal positive and the maximal eigenvalues of a matrix $A$ respectively.
\item $\smin(A)$, $\sminp(A),~\smax(A)$ are the minimal, the minimal positive and the maximal singular values $\sigma_i(A) = \sqrt{\lambda_i(A^\top A)}$ of a matrix $A$ respectively.
\end{itemize}

\textbf{Notation}
\begin{itemize}
\item ``Dimension-lifted'' vectors and matrices are written in bold: $\bx, \bA$.
\item $x_{[i]}$ denotes the $i$-th component of a vector $x$.
\item The identity matrix of order $d$ is denoted by $I_d$. Sometimes the subscript is omitted if the order of $I$ is clear from the context.
\item $\one_d$ denotes the column vector of ones in $\R^d$.
\item $\image A$ and $\ker A$ denote the image and the null space of a linear operator $A$ respectively.
\end{itemize}

\section{Problem formulation}
\subsection{Objective and constraints}
We consider the following affine constrained optimization problem, where $f_i$ are assumed to be $L_F$-smooth (have Lipschitz-continuous gradient with constant $L_F$)  and $\mu_F$-strongly convex:
\begin{align}\label{eq:lifted}
	\min_{x_1,\ldots,x_n\in\R^d}~ &\sum_{i=1}^n f_i(x_i) \\
	\text{s.t. } &A_i x_i = b_i,~ i = 1, \ldots, n \\
	&x_1 = \ldots = x_n.\label{eq:consensus}
\end{align}
Practical examples of this type of finite-sum affine constrained optimization problems include constrained estimation problems, such as constrained least squares problems \cite{zhou2013path}.

The matrix-vector form of the problem is
\begin{align}\label{eq:common_constraints}
	\min_{\bx\in\R^{nd}}~ &F(\bx) \\
	\text{s.t. } &\bA \bx = \bb \nonumber \\
	&\bx \in \cL \nonumber,
\end{align}
where 
$\bx$ is the column vector of $x_1, \ldots, x_n$, $x_i \in \R^d$,
$F(\bx) = \sum_{i=1}^n f_i(x_i)$, and
$\cL$ denotes the \textit{consensus hyperlane} defined by constraint \eqref{eq:consensus}.
Note, that $F(\bx)$ is also $\mu_F$-strongly convex and $L_F$-smooth.

In case when $A_i$ and $b_i$ are all different, we set 
$\bA = \diag\cbraces{A_1, \ldots, A_n}$, $\bb = \one \otimes b$, where $\otimes$ denotes the Kronecker product.
If all $A_i = A$ and $b_i = b$ are the same, 
then there are several variants to choose $\bA, \bb$, e.g.:
\begin{itemize}
	\item $\bA = \bI\otimes A = \diag\cbraces{A, \ldots, A}$,
  $\bb = \one\otimes b$, or
	\item $\bA = (A~ 0~ \ldots~ 0),~ \bb = b$.
\end{itemize}
This logic also applies if there are clusters of agents with same affine constraints in each group.
For definiteness, we will assume that the first variant is chosen. 

\subsection{Decentralized communication}
We assume that the problem is distributed over a computational network consisting of $n$ agents (or nodes).
Each agent privately holds $f_i$, $A_i$ and $b_i$.
Agents are connected through a communication network, represented by a time-varying undirected graph, i.e. a sequence of undirected graphs. 
Agents are only allowed to exchange information with their current neighbors in the communication graph
The limitations imposed on the communication process are formally described in Defintion~\ref{def:foda}.

In further developments, we will heavily rely on the notion of
\textit{gossip matrix}.
$W$ is a gossip matrix of an undirected graph $G = (V, E)$, $|V| = n$ if it satisfies following properties 
\begin{enumerate}
  \item ${W}$ is symmetric and positive semi-definite;
	\item (Network compatibility) $[{ W}]_{ij} \neq 0$ if and only if $(i, j) \in E$ or $i = j$;
	\item (Kernel property) For any $v = [v_1,\ldots,v_n]^\top\in\R^n$, ${ W} v = 0$ if and only if $v_1 = \ldots = v_n$.
\end{enumerate}
A typical example of a gossip matrix is the \textit{Laplacian matrix} of a graph: $W\in \mathbb{R}^{n\times n}$,
\begin{align}\label{eq:laplacian}
[W]_{ij} = \begin{cases}
-1,  & \text{if } (i,j) \in E,\\
\text{deg}(i), &\text{if } i = j, \\
0,  & \text{otherwise.}
\end{cases}
\end{align}

Later we will use the dimension-lifted gossip matrix $W \otimes I_d$.
From the third property of gossip matrices, we have that $\cbraces{W \otimes I_d} \bx = 0$ if and only if $\bx\in \cL$.

Since we assume that the communication network is time-varying, we denote the communication graph at the $k$-th step by $G(k) = (V, E(k))$ and the associated gossip matrix by $W(k)$.
Note that the existence of a gossip matrix for each step implies, by the kernel property, that $G(k)$ is connected for all $k$. 
According to the second property of gossip matrices, multiplication $\cbraces{W(k) \otimes I_d} \bx$ can be performed in a decentralized way at step $k$.

The convergence rate of decentralized optimization algorithms depends on the spectrum of the gossip matrices.
Therefore we assume that
\begin{equation}\label{eq:gossip-spectrum}
\tilde \lambda_{\min}^+\leq \lambda_{\min}^+(W(k))\leq \lambda_{\max}(W(k))\leq \tilde \lambda_{\max}.
\end{equation}

\section{Lower bounds}
We consider the class of \textit{first-order decentralized algorithms} defined as follows.
\begin{definition}\label{def:foda}
Denote $\cM_i(k)$, where $\cM_i(0) = \{x_i^0\}$, as the local memory of the $i$-th agent at step $k$. 
The set of allowed actions of a first order decentralized algorithm at step $k$ is restricted to the three options
\begin{enumerate}
\item Local computation: $\cM_i(k) = \spn\cbraces{\{x, \nabla f_i(x), \nabla f^*_i(x): x \in \cM_i(k) \}}$;
\item Decentralized communication: $\cM_i(k) = \spn\cbraces{\{\cM_j(k): \text{edge}~ (i, j) \in E(k)\}}$.
\item Matrix multiplication: $\cM_i(k) = \spn\cbraces{\{b_i, A_i^\top A_i x: x \in \cM_i(k)\}}$.

\end{enumerate}
After each step $k$, an algorithm must provide a current approximate solution $x^k_i \in \cM_i(k)$ and set $\cM_i(k+1) = \cM_i(k)$.
\end{definition}

Using the standard approach for constructing lower bounds in smooth strongly convex optimization \cite{nesterov2004introduction}, we consider $d=\infty$.
Let $m$ be a positive integer parameter, $m \leq d$. Then, following \cite{salim2022optimal}, we consider the affine-constrained problem with constraint $(W' \otimes I_{d/m})x = 0$, where $W'\in \R^{m\times m}$ is a gossip matrix of some (static) communication graph, thus interpreting an affine-constrained problem as a decentralized optimization problem.

Let $W'$ be a gossip matrix associated with the graph $G'(V', E')$, $V' = \{1, \ldots, m\}$. 
Set $A_i = A = \sqrt{W' \otimes I_{d/m}}$ and $b_i = 0$. 
This leads to a two-level decentralized optimization problem.
On the upper level we have the conventional decentralized optimization problem over the communication network $G$:
\begin{align*}
	&\min_{x\in\R^d} \sum_{i=1}^n f_i(x),
\end{align*}
but each $f_i = \sum_{j=1}^m f_{ij}$ is distributed among the subnodes of the inner computational network $G'$ located inside the node $i$. 
This forms the inner level of our problem, as shown in Fig. \ref{fig:2graph}.
Thus, instead of thinking about affine constraints, we can think about the inner decentralized computational network.
From this perspective, subnodes exchange the $d/m$-dimension vectors, and nodes exchange the $d$-dimension vectors, which are stacked from the vectors of their subnodes.

We use this construction to obtain lower bounds for both static and time-varying setups.
\begin{theorem}\label{thm:lower_static}
For every $L_F \geq \mu_F >0$, $\chi_W >0$ and $\chi_A > 0$ there exists a decentralized optimization problem \eqref{eq:lifted} (\eqref{eq:common_constraints}) over a static communication graph $G$
with a gossip matrix $W: \frac{\lmax(W)}{\lminp(W)} = \chi_W$, and a matrix $\mA: \frac{\lmax(\mA^\top \mA)}{\lminp(\mA^\top \mA)} = \chi_A$, such that any decentralized first-order algorithm, as per Definition~\ref{def:foda}, requires at least
\begin{align*}
    &N = \Theta\cbraces{\sqrt{\frac{L_F}{\mu_F}} \log\frac{1}{\varepsilon}} \text{sequential local computations}, \\
    &\Theta\cbraces{N \sqrt{\chi_W}} \text{communications}, \\
    &\Theta\cbraces{N\sqrt{\chiA}} \text{multiplications by $\mA^\top \mA$}.
\end{align*}
to achieve 
$\sqn{x_i^N - x^*} \leq \eps~\forall i \in [1, n] \cap \N$. 
\end{theorem}

\tikzset{inner/.pic={
    \begin{scope}
    \def \l{0.3}
    \tikzstyle{every node}=[circle, draw, fill=black,
                        inner sep=0pt, minimum width=4pt]
    \draw (-\l,-\l) node{} -- (0,0) node {} -- ++({\l * sqrt(2)}, 0) node{};
    \end{scope}
}}

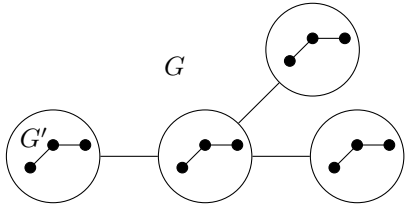
\begin{figure}
\centering
\begin{tikzpicture}
    \def \l{2}
    \begin{scope}
    \def \s{0.15}
    \tikzstyle{every node}=[circle, draw, inner sep=0pt, minimum width=35pt]
    \node (1) at (0,0) {};
    \pic at (0,\s) {inner} ;

    \node (2) at (\l,0) {};
    \pic at (\l,\s) {inner} ;

    \node (3) at ({\l * (1 + 1/sqrt(2))}, {\l / sqrt(2)}) {};
    \pic at ({\l * (1 + 1/sqrt(2))}, {\l / sqrt(2) + \s}) {inner} ;

    \node (4) at ({2 *\l},0) {};
    \pic at ({2 *\l}, \s) {inner} ;

    \draw (1) -- (2) -- (3);
    \draw (2) -- (4);
    \end{scope}

    \draw (-0.25, 0.25) node {$G'$};
    \draw (\l * 0.8, {\l * 0.6 }) node {$G$};
\end{tikzpicture}
\vspace{-0.5cm}
\caption{Two-level computational network}
\label{fig:2graph}
\end{figure}
The proof is based on the technique from \cite{scaman2017optimal} and is provided in the Appendix.

We now present a variant of Theorem~\ref{thm:lower_static} for the time-varying communication networks. 
\begin{theorem}\label{thm:lower_time-varying}
For every $L_F \geq \mu_F >0$, $\chi_A > 0$, $\tilde\lambda_{\min}^+$ and $\tilde\lambda_{\max}:$ $\frac{\tilde\lambda_{\max}}{\tilde\lambda_{\min}^+} = \chi_W \geq 3$ there exists a decentralized optimization problem \eqref{eq:lifted} (\eqref{eq:common_constraints}), 
a sequence of gossip matrices $W(k)$ satisfying \eqref{eq:gossip-spectrum}, and a matrix $\mA: \frac{\lmax(\mA^\top \mA)}{\lminp(\mA^\top \mA)} = \chi_A$, such that  any decentralized first-order algorithm, as per Definition~\ref{def:foda}, requires at least
\begin{align*}
    &N = \Theta\cbraces{\sqrt{\frac{L_F}{\mu_F}} \log\frac{1}{\varepsilon}} \text{sequential local computations}, \\
    &\Theta\cbraces{N \chi_W} \text{communications}, \\
    &\Theta\cbraces{N\sqrt{\chiA}} \text{multiplications by $\mA^\top \mA$}.
\end{align*}
to achieve 
$\sqn{x_i^N - x^*} \leq \eps~\forall i \in [1, n] \cap \N$. 
\end{theorem}
The proof is based on \cite{kovalev2021lower} and is provided in the Appendix.

\section{Application of ADOM}
By standard duality arguments we rewrite problem~\eqref{eq:common_constraints} as
\begin{align*}
	\min_{\bx} &\max_{\bs \in \cLp,\bp} \sbraces{F(\bx) - \angles{\bs, \bx} - \angles{\bp, \bA \bx - \bb}} \\
	&= \max_{\bs \in \cLp ,\bp} \sbraces{-\max_\bx\sbraces{\angles{\bs + \bA^\top \bp, \bx} - F(\bx)} + \angles{\bp, \bb}} \\
	&= \max_{\bs \in \cLp,\bp} \sbraces{-F^*(\bs + \bA^\top \bp) + \angles{\bp, \bb}} \\
	&= -\min_{\bs \in \cLp ,\bp} \sbraces{F^*(\bs + \bA^\top \bp) - \angles{\bp, \bb}},
\end{align*}
where $F^*$ is the convex (Fenchel) conjugate of $F$.
Therefore, problem~\eqref{eq:common_constraints} is equivalent to
\begin{align*}
 \min_{\bs \in \cLp ,\bp} \sbraces{F^*(\bs + \bA^\top \bp) - \angles{\bp, \bb}}.
\end{align*}
Introduce 
\begin{align*}
\bz = \begin{pmatrix} \bp \\ \bs \end{pmatrix},
~ \bq = \begin{pmatrix} \bb \\ 0 \end{pmatrix},
~ \bB = \begin{pmatrix} \bA \\ I_{nd} \end{pmatrix},
~ \bP = \begin{pmatrix} I_{\dim{(\bb)}} & 0 \\ 0 & I_{nd} - \frac{1}{n}\cbraces{\one_n\one_n^\top}\otimes I_d \end{pmatrix},~
\\
\bW(k) = \begin{pmatrix} I_{\dim(\bb)} & 0 \\ 0 & W(k)\otimes I_d \end{pmatrix},
\end{align*}
where $\dim(\bb)$ is the number of components in vector $\bb$.
Also denote $H(\bz) = F^*(\bB^\top \bz) - \angles{\bz, \bq}$. Now we can equivalently rewrite optimization problem~\eqref{eq:common_constraints} as
\begin{align}\label{eq:problem_common_constraints_dual}
\min_{\bz\in\image\bP}~ H(\bz).
\end{align}
After that, we apply ADOM~\cite{kovalev2021adom} to the problem~\eqref{eq:problem_common_constraints_dual}.
\begin{algorithm}[ht]
\caption{ADOM with local affine constraints}
\label{alg:adom_common_constraints}
\begin{algorithmic}[1]
	\REQUIRE{$\bz^0\in\image\bP\bB,~ \bm^0\in\image\bP\bB,~ \alpha, \eta, \theta, \sigma> 0,~ \tau\in (0, 1)$}
	\STATE{Set $\bz_f^0 = \bz^0$}
	\FOR{$k = 0, 1, \ldots$}
	\STATE{$\bz_g^k = \tau \bz^k + (1 - \tau)\bz_f^k$}\label{adom:line:zg}
	\STATE{$\Delta^k = \sigma\bW(k)(\bm^k - \eta\nabla H(\bz_g^k))$}\label{adom:line:delta}
	\STATE{$\bm^{k+1} = \bm^k - \eta\nabla H(\bz_g^k) - \Delta^k$}\label{adom:line:m}
	\STATE{$\bz^{k+1} = \bz^k + \eta\alpha (\bz_g^k - \bz^k) + \Delta^k$}\label{adom:line:z}
	\STATE{$\bz_f^{k+1} = \bz_g^k - \theta\bW(k)\nabla H(\bz_g^k)$}\label{adom:line:zf}
	\ENDFOR
\end{algorithmic}
\end{algorithm}

Note that $b_i\in \image A_i$ and therefore $\bq\in\image \bB$. Therefore, for any $\bz\in\R^{nd}$ we have $\nabla H(\bz) = \bB\nabla F^*(\bB^\top \bz) - \bq\in\image\bB$. We conclude that the iterates $\bz^k, \Delta^k, \bz_f^k, \bz_g^k$ of Algorithm~\ref{alg:adom_common_constraints} lie in $\image\bP\bB = \image\bA\otimes\cL^\bot$, and $\bm^k\in\image\bB = \image\bA\otimes\R^{nd}$.

On the subspace $\image\bP\bB$ we estimate strong convexity and smoothness constants as
\begin{align*}
	\mu_H = \frac{1 + (\sminp(A))^2}{L_F},~ L_H = \frac{1 + \smax^2(A)}{\mu_F}.
\end{align*}
From \eqref{eq:gossip-spectrum}  we have
\begin{equation}\label{eq:spectrum}
\lambda_{\min}^+ \defeq \min(1, \tilde\lambda_{\min}^+) \leq \lambda_{\min}^+(\mW(k))\leq \lambda_{\max}(\mW(k))\leq \max(1, \tilde\lambda_{\max}) \eqdef \lambda_{\max}.
\end{equation}

Now we can formulate the key convergence result. 
\begin{theorem}\label{thm:main}
	Set parameters $\alpha, \eta, \theta, \sigma,\tau$ of Algorithm~\ref{alg:adom_common_constraints} to $\alpha = \frac{\muH}{2}$, $\eta = \frac{2\lminp}{7\lmax\sqrt{\muH \LH}}$, $\theta = \frac{1}{\LH \lmax}$, $\sigma = \frac{1}{\lmax}$, and $\tau = \frac{\lminp}{7\lmax}\sqrt{\frac{\muH}{\LH}}$. 
Then there exists $C>0$, such that for Algorithm~\ref{alg:adom_common_constraints} applied for problem~\eqref{eq:common_constraints} it holds
	\begin{equation}
	\squeeze 
		\sqn{\g F^*(\mB^\top \bz_g^N) - \bx^*} \leq C \left(1- \frac{\lminp}{7\lmax} \sqrt{\frac{\muH}{\LH}}\right)^N,
	\end{equation}
\end{theorem}
where $\bx^*$ is the solution of the problem~\ref{eq:common_constraints}.

The proof is a minor modification of original ADOM convergence proof in \cite{kovalev2021adom}, and is provided in the appendix for reader's convenience.

As a corollary of Theorem \ref{thm:main} we have the following communication, dual-oracle call and matrix multiplication complexity: 
\begin{align*}
	N \leq O\cbraces{\frac{\max(1, \tilde\lambda_{\max})}{\min(1,\tilde\lambda_{\min}^+) }\sqrt{\frac{L_F}{\mu_F}}\sqrt\frac{1+ \sigma_{\max}^2(A)}{1 + (\sigma_{\min}^+(A))^2}\log\cbraces{\frac{1}{\eps}}},
\end{align*}
where $\eps$ is the desired accuracy of the approximate solution: $\sqn{\g F^*(\mB^\top \bz_g^N) - \bx^*} \leq \eps$.

\section{Separating complexities}\label{sec:separate}
As shown in \cite{kovalev2021lower}, one can separate oracle (computation) and communication complexities in the time-varying setup by using the multi-consensus procedure.
This will not change (up to the $\ln 2$ factor) the number of communications but reduces the number of oracle calls.

In the time-varying setup, acceleration over $\chi = \frac{\lmax}{\lminp}$ is not applicable, as stated by Theorem~\ref{thm:lower_time-varying}, so we can separate communication complexity but can not improve it.
However, since $\mA$ is constant, we can use Chebyshev acceleration over $\mA$ to separate matrix multiplication complexity and decrease the number of multiplications by $\mA$, as was done for communication complexity in the static communication graph setup in \cite{scaman2017optimal}.

For this section we will assume that $W(k)$ are divided by their maximum eigenvalue and therefore $\lmax = \tilde\lambda_{\max}  = 1$ and $\lminp = \tilde\lambda_{\min}^+$.
In practice this could be achieved by using $I - M(k)$ instead of $W(k)$, there $M(k)$ are the Metropolis weight matrices \cite{kovalev2021lower}.

Let $P(x)$ be a polynomial such that $P(0) = 0$ and $P(\lambda_i) \neq 0$ for all eigenvalues $\lambda_i$ of $A$.
Then, using the fact that $\ker M = \ker M^T M$, we can do the following sequence of equivalent reformulations
\begin{align*}
    \mA \bx = \bb  
    &\Leftrightarrow 
    \mA \bx = \mA \bx_0  
    \Leftrightarrow 
    \mA^\top \mA \bx = \mA^\top \mA \bx_0
    \\&\Leftrightarrow  
    P(\mA^\top \mA) \bx = P'(\mA^\top \mA)\mA^\top \mA \bx_0
    \\&\Leftrightarrow  
    P(\mA^\top \mA) \bx = P'(\mA^\top \mA)\mA^\top \bb.
\end{align*}
where $\bx_0$ is any vector satisfying $\mA \bx_0 = \bb$ (here we used consistency of constraints),
and $P'(x) = P(x)/x$ is correctly defined since $P(0) = 0$.

Thus the idea is to replace $\mA$ with $P(\mA^\top \mA)$ to improve the spectral properties of the matrix.
The polynomials of choice are shifted and scaled Chebyshev polynomials \cite{scaman2017optimal}, because Chebyshev polynomials increase magnitude more quickly than any other polynomials of the same degree satisfying $|P(x)| \leq 1~ \forall x \in [-1,1]$.
This allows to significantly compress the spectrum, using polynomials of a relatively low degree.
In particular, let $P(x)$ be defined as
\begin{align*}
    P(x) = 1 - \frac{T_K\left(-\nu + \frac{2}{\lmax - \lminp}x\right)}{T_K(-\nu)},~ K = \floor{\sqrt \chiA},
\end{align*}
where $\chiA = \smax(\mA^T \mA) / \sminp(\mA^T \mA)$, $\nu = \frac{\chi +1}{\chi -1}$ and $T_K$ are Chebyshev polynomials defined by $T_0(x) = 1$, $T_1(x) = x$ and $T_{K+1} = 2x T_K(x) - T_{K-1}(x)$ for $K \geq 1$.
If $P$ has degree $K = \floor{\sqrt \chiA}$, then $\chi\left( P(\mAT \mA)\right) \leq 4$ \cite{scaman2017optimal}, which allows to quadratically improve the dependence of convergence rate on $\chi_A$
by replacing $\mA$ with $P(\mA^\top \mA)$ and $\bb$ with $P'(\mAT \mA)  \mAT \bb$.

To separate communication complexity, a multi-consensus procedure should be used, i.e. $W(k)$ should be replaced with $D(W(k)) = I - (I - W(k))^{K}$, where $K = \ceil{\chi \ln 2}$, what makes $\chi = O(1)$ at the cost of $K$ communication rounds \cite{kovalev2021lower}.

Finally, by replacing $\mA \to P(\mAT \mA)$, $\bb \to P'(\mAT \mA)  \mAT \bb$ and $W(k) \to D(W(k))$ we obtain following complexity estimates to reach $\sqn{\g F^*(\mB^\top \bz_g^k) - \bx^*} \leq \eps$ for ADOM algorithm with affine constraints, Chebyshev acceleration, and multi-consensus:
\begin{align*}
    &N = O\cbraces{\sqrt{\frac{L_F}{\mu_F}} \log\frac{1}{\varepsilon}} \text{ oracle calls at each node}, \\
    &O\cbraces{N \chi} \text{ communications}, \\
    &O\cbraces{N\sqrt{\chiA}} \text{ multiplications by $\mA^\top \mA$}.
\end{align*}
These upper bounds match the lower bounds of Theorem~\ref{thm:lower_time-varying}, thus the obtained algorithm is optimal among first-order decentralized algorithms for strongly convex problems with affine constraints on time-varying networks.

\section{Numeric validation}
We verify Theorem~\ref{thm:main} with numeric experiments\footnote{Source code: \href{https://github.com/niquepolice/ADOM_affine_constraints}{https://github.com/niquepolice/ADOM\_affine\_constraints}}
on problems with the quadratic objective:

\begin{align*}
&\sum_{i=1}^n f_i(x) \to \min_x
\\\text{s.t. }& Ax = b,
\\f_i(x) &= \frac12 x^T C_i x + d_i^T x, \;  \mu_F I \preceq C_i \preceq L_F I.
\end{align*}

From the results of Section~\ref{sec:separate}, the influence of the gossip matrix's spectrum and the affine constraint matrix spectrum on the convergence rates of Algorithm \ref{alg:adom_common_constraints} is straightforward to comprehend.
Therefore, we focus only on the impact of the objective's condition number on the convergence rates.

Our numerical experiments are not designed to simulate real-world problems; rather, they serve as illustrations of the algorithm's theoretical properties.
This is because quadratic objectives are good representatives of the smooth and strongly convex problem class.

It is not difficult to implement exact dual oracle for this objective, but we do not want to exploit the simplicity of quadratic problem, and, following \cite{kovalev2021adom}, we obtain an approximation of $\bg^k \approx \nabla F^*(\mBT \bz_g^k)$ by using few gradient steps at each iteration: $\bg^{k} = \bg^{k-1} - \frac{1}{L_F}(\nabla F(\bg^{k-1}) - \mBT \bz_g^k)$. 
So in fact the implemented algorithm uses a primal oracle because dual oracle call in Algorithm \ref{alg:adom_common_constraints} is replaced with primal oracle call.

Experiment parameters are $d = 20$, $A \in \R^{10 \times d}$, $\chiA = 20$, $\mu_F=1$.
Communication graphs $G(k)$ are random ring graphs at each iteration.
We run Algorithm~\ref{alg:adom_common_constraints} for $N=2500$ iterations for different values of $L_F$, and do the linear regression to obtain the coefficient $\kappa$ in the dependence $\|\bg^k - \bx^*\|_2 = C_1 \exp(-\kappa N)$ using only last $N/2$ iterations.
This is illustrated in Figure~\ref{fig:conv}.
In all cases a steady linear convergence to the solution is present.

Then we do the linear regression in the log-log scale to obtain the coefficient $\nu$ in the dependence
$\kappa = C_2 \cbraces{\frac{L_F}{\mu_F}}^\nu$, as shown in Figure~\ref{fig:nu}.
The resulting value is $\nu \approx 0.54$ with the standard error of $\approx 0.03$, which is rather close to the value $\nu= \frac12$ in the Theorem~\ref{thm:main}.

\begin{figure}
    \centering
    \includegraphics[width=0.9\textwidth]{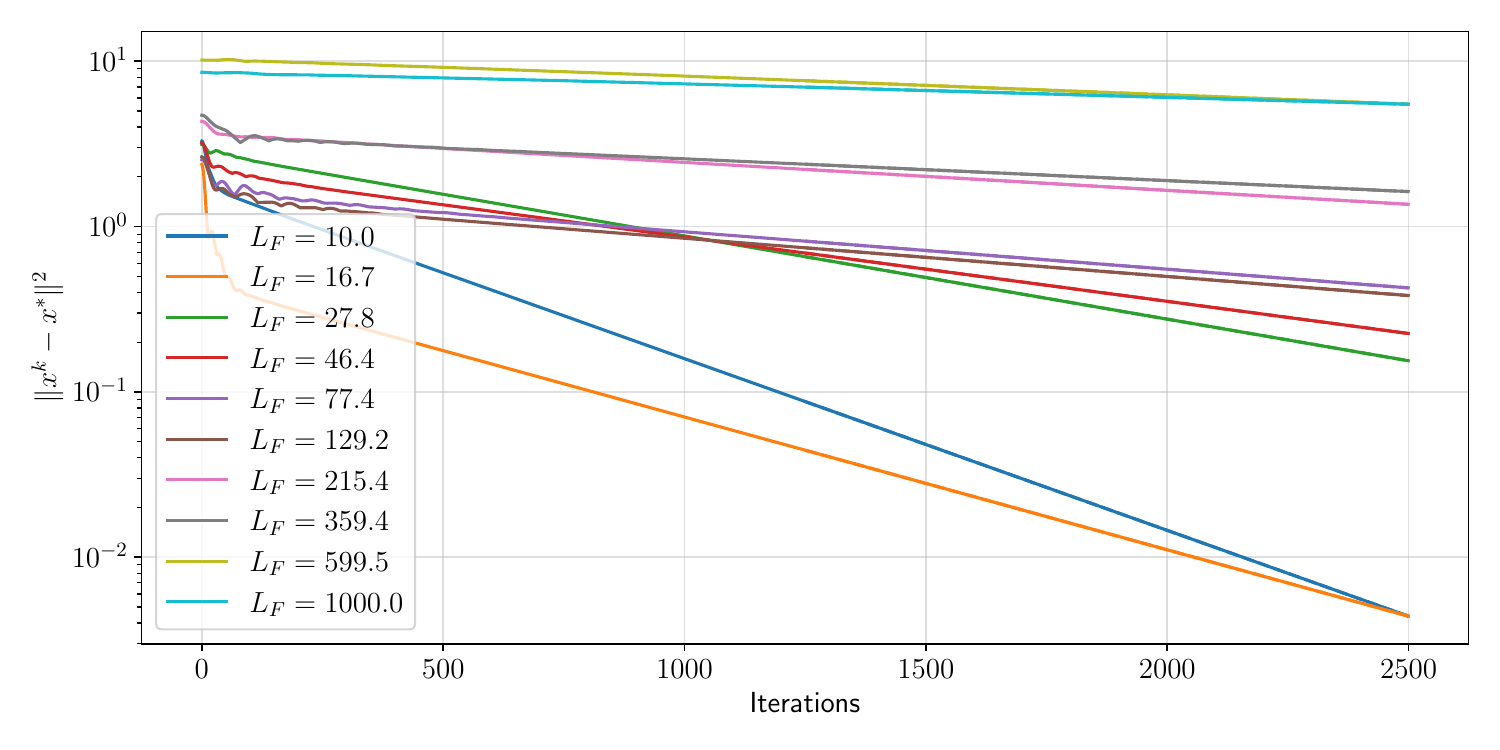}
    \caption{Algorithm \ref{alg:adom_common_constraints} convergence for different values of $L_F$}
    \label{fig:conv}
\end{figure}

\begin{figure}
    \centering
    \includegraphics[width=0.6\textwidth]{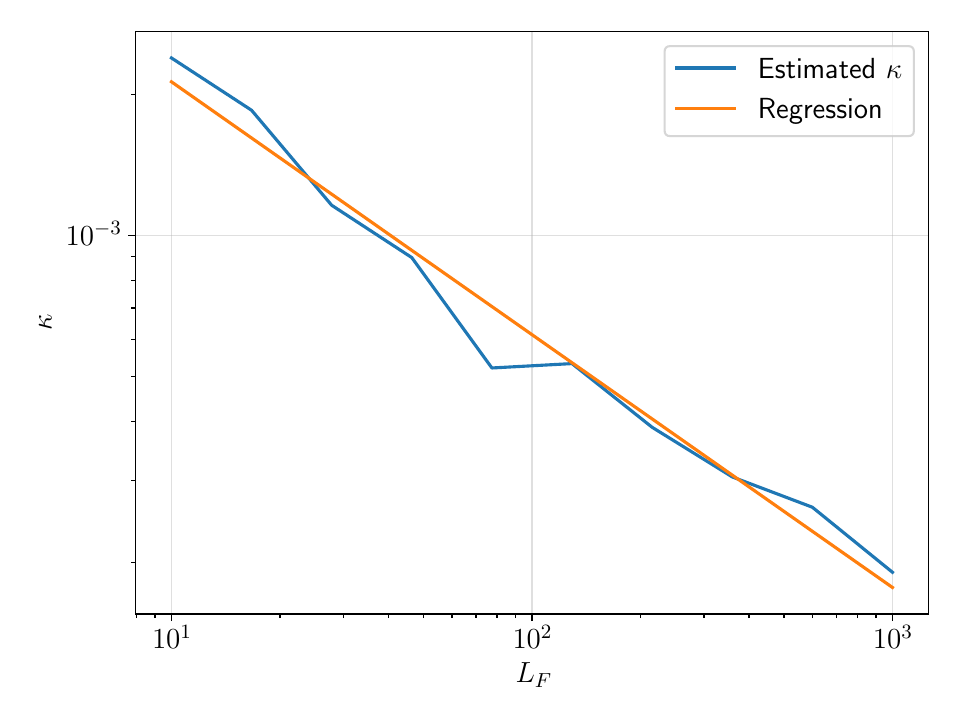}
    \caption{Regression for $\nu$}
    \label{fig:nu}
\end{figure}

\section{Conclusion}
By viewing the affine-constrained problem as a decentralized optimization problem alike \cite{salim2022optimal}, and combining constructions of lower bound for static \cite{scaman2017optimal} and time-varying \cite{kovalev2021lower} setups, lower bounds for decentralized optimization with affine constraints over static and time-varying networks via first-order methods were obtained.

As we found, the ADOM algorithm can be straightforwardly extended to the affine-constrained case.
For this problem class, we were also able to apply Chebyshev acceleration over $A$, and the resulting complexity estimates match the lower bounds.

However, a lot of questions are left for the future work.
We did not succeed to provide an extension of the ADOM+ algorithm \cite{kovalev2021lower} to the affine-constrained problems, thus no linearly convergent primal algorithm is known for this problem class .
It is also of interest to obtain optimal algorithms for time-varying networks in case of shared affine inequality constraints $\sum_i \cbraces{A_i x_i - b_i}\leq 0$, where $A_i$ and $b_i$ are held privately by $i$-th agent. This problem variant has more practical applications \cite{wang2022distributed,yarmoshik2022decentralized,necoara2011parallel}, but also brings additional difficulties to the theoretical analysis, e.g. in this case $b_i$ might not belong to $\image A_i$. 
This means that we cannot apply our approach, because it requires a gradient method to stay in the subspace where the objective is strongly convex.

\section{Acknowledgements}

This work was supported by a grant for research centers in the field of artificial intelligence, provided by the Analytical Center for the Government of the Russian Federation in accordance with the subsidy agreement (agreement identifier 000000D730321P5Q0002) and the agreement with the Moscow Institute of Physics and Technology dated November 1, 2021 No. 70-2021-00138.
\section{Appendix}


\subsection{Proof of Theorem~\ref{thm:lower_static}}

\begin{proof}
Let the affine constraint in problem \ref{eq:lifted} be $A_i x_i = 0$, with $A_i=A=\sqrt{W' \otimes I_{d/m}}$.
Then the affine constrained decentralized problem can be seen as two-level decentralized problem, as explained above.

Select sets of subnodes $S_1$, $S_2$ and $S_3$ such that $S_1$, $S_2$ are at the distance~$\geq \Delta_A$ through the inner graph,
and $S_2$, $S_3$ are at the distance~$\geq \Delta_W$  through the outer graph.
Consider the following splitting of the Nesterov's ``bad'' function
\begin{align}\label{eq:splitting}
f_{ij}(x) = 
\frac\alpha{2mn} \sqn{x}
+ 
\frac{\beta - \alpha}{8} \cdot
  \begin{cases}
  \frac1{|S_1|}\cbraces{x^\top M_1 x - 2x_{[1]}},~ (i,j) \in S_1,\\
  \frac1{|S_2|}x^\top M_2 x,~ (i,j) \in S_2,\\
  \frac1{|S_3|}x^\top M_3 x,~ (i,k) \in S_3,\\
  0,~ \text{otherwise,}
  \end{cases}
\end{align}
where 
\begin{align*}
M_1 &= \diag{(1, 0, M_0, 0, M_0, \ldots)},\\
M_2 &= \diag{(M_0, 0, M_0, 0, \ldots)},\\
M_3 &= \diag{(0, M_0, 0, M_0, \ldots)},\\
M_0 &= \begin{pmatrix}
1 &-1\\
-1& 1 \end{pmatrix}.
\end{align*}


Then increasing the number of nonzero components in $x^k$ on any subnode by three requires one local computation on a node in $S_1$, $\Delta_A$ inner communications a.k.a. multiplications by $A^\top A$, one local computation on a node in $S_2$, $\Delta_W$ communications in the outer graph and one local computation on a node in $S_3$.
Denote by $\kappa_g = \frac{\beta}{\alpha}$ the ``global'' condition number of $\sum_{ij}^{mn}f_{ij}$.
Since the solution is 
$x^*_k = \left(\frac{\sqrt \kappa_g - 1}{\sqrt \kappa_g +1 }\right)^k$, we have
\begin{equation}\label{eq:nesterov-lower}
\sqn{x^N - x^*} \geq \sum_{k=N+2}^\infty (x^*_k)^2 \geq \left(\frac{\sqrt \kappa_g - 1}{\sqrt \kappa_g +1 }\right)^{N+2},
\end{equation}
where $N$ is the number of iterations, each including 3 sequential computational steps, $\Delta_A$ multiplications by $A^TA$ and $\Delta_W$ communications.


To finish the proof we need to construct communication graphs $G$, $G'$, where distances between $S_1$, $S_2$ and $S_2$, $S_3$ are close to $\Delta_A, \Delta_W$, and equip the graphs with gossip matrices with given condition numbers $\chi_A, \chi_W$.

We should also choose $\alpha$ and $\beta$ such that $f_i$ are $L_F$-smooth and $\mu_F$-strongly convex,
and choose $S_1$, $S_2$, $S_3$ so that $\kappa_g$ is similar to $\frac{L_F}{\mu_F}$. 

Denote $\gamma(M) = \sigma_{\min}^{+}(M) / \sigma_{\max}(M)$, $\gamma_W = 1/\chi_W$, $\gamma_A = 1/\chi_A$.
Let $\gamma_n=\frac{1-\cos \left(\frac{\pi}{n}\right)}{1+\cos \left(\frac{\pi}{n}\right)}$ be a decreasing sequence of positive numbers. 
Since $\gamma_2=1$ and $\lim _n \gamma_n=$ 0, there exists $n \geq 2$ such that $\gamma_n \geq \gamma >\gamma_{n+1}$ and 
$m \geq 2$ such that $\gamma_m \geq \gamma >\gamma_{m+1}$.

First, construct graph $G$.
The cases $n=2$ and $n \geq 3$ are treated separately.
If $n \geq 3$, let $G$ be the linear graph of size $n$ ordered from node $1$ to $n$, and weighted with 
$w_{i, i+1}=\begin{cases}1-a,  i=1\\1, \text{otherwise}\end{cases}.$
Then set 
$S_{2G}=\left\{1, \ldots, \lceil n / 32\rceil\right\}$ and $\Delta_W=(1-1 / 16) n-1$, so that 
$S_{3G} = \{\lceil n / 32\rceil + \lceil \Delta_W\rceil, \ldots, n\}$.


Take $W_a$ as the Laplacian of the weighted graph $G$. 
A simple calculation gives that, if $a=0$, $\gamma\left(W_a\right)=\gamma_n$ and, if $a=1$, the network is disconnected and $\gamma\left(W_a\right)=0$.
Thus, by continuity of the eigenvalues of a matrix, there exists a value $a \in[0,1]$ such that $\gamma\left(W_a\right)=\gamma_W$.
Finally, by definition of $n$, one has $\gamma_W>\gamma_{n+1} \geq \frac{2}{(n+1)^2}$, and $\Delta_W \geq \frac{15}{16}\left(\sqrt{\frac{2}{\gamma_W}}-1\right)-1 \geq \frac{1}{5 \sqrt{\gamma_W}}$ when $\gamma_W \leq \gamma_3=\frac{1}{3}$.

For the case $n=2$, we consider the totally connected network of 3 nodes, reweight only the edge $\left(1, 3\right)$ by $a \in[0,1]$, and let $W_a$ be its Laplacian matrix.
If $a=1$, then the network is totally connected and $\gamma\left(W_a\right)=1$.
If, on the contrary, $a=0$, then the network is a linear graph and $\gamma\left(W_a\right)=\gamma_3$.
Thus, there exists a value $a \in[0,1]$ such that $\gamma\left(W_a\right)=\gamma$.
Set $S_{2G}=\left\{1\right\}, S_{3G}=\left\{2\right\}$, then $\Delta_W=1 \geq \frac{1}{\sqrt{3 \gamma_W}}$.

Second, do the same for graph $G'$, obtaining $m$, $S_{1G'}, S_{2G'}$ and $\Delta_A \geq \frac{1}{5\sqrt{\gamma_A}}$.

Define $S_1 = S_{2G} \times S_{1G'}$, $S_2 = S_{2G} \times S_{2G'}$ and $S_3 = S_{3G} \times S_{2G'}$, see Fig. \ref{fig:func-splitting}.
In all cases we have $|S_k| \geq |S_2| \geq \lceil\frac{n}{32}\rceil \lceil \frac{m}{32} \rceil$ for $k\in\{1,3\}$.

Because $\mu_F = \frac{\alpha}{n}$, we set $\alpha = \mu_F n$.
Since $0 \preceq M_k \preceq 2I$ for $k\in\{1,2,3\}$, $L_F = \frac{\alpha}{n} + \frac{(\beta-\alpha)m}{2|S_2|}$, thus set $\beta=2|S_2|(L_F-\mu_F)/m + \mu_F n$ to make all $f_i$ be $L_F$-smooth and $\mu_F$-strongly convex.
Then $\kappa_g = \frac{\beta}{\alpha} = 1 + \frac{2|S_2|(L_F - \mu_F)}{\mu_F mn} \geq \frac{L_F}{512 \mu_F}$.
Combining this with \eqref{eq:nesterov-lower} and the inequalities between $\Delta_A, \gamma_A$ and $\Delta_W, \gamma_W$ we conclude the proof. 
\end{proof}

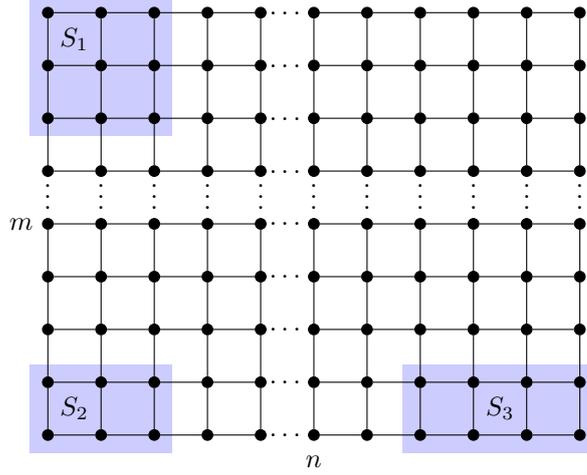
\begin{figure}[h]
\centering
\begin{tikzpicture}
    \def \m{8}
    \def \n{10}
    \def \l{0.7}
    \def \ma{\l/3} 
    \begin{scope}

    \fill [fill=blue!20,thick]
      (-\ma,-\ma) rectangle ({2 * \l + \ma}, {1 * \l + \ma});

    \fill [fill=blue!20,thick]
      (-\ma,{\m * \l + \ma}) rectangle ({2 * \l + \ma}, {(\m - 2) * \l - \ma});

    \fill [fill=blue!20,thick]
      ({\n * \l + \ma}, -\ma) rectangle ({(\n - 3) * \l - \ma}, {1 * \l + \ma});

    \tikzstyle{every node}=[circle, draw, fill=black,
                        inner sep=0pt, minimum width=4pt,draw]
    \foreach \ni in {0,...,\n}
    \foreach \mi in {0,...,\m} 
        \node (\ni x \mi) at ({\ni * \l}, {\mi * \l}) {};

    \end{scope}

    \def \st{}
    \foreach \ni in {1,...,\n}
    \foreach \mi in {0,...,\m} 
      \pgfmathsetmacro\k{int(\ni-1)}
      \ifthenelse{\ni=5}{\path (\k x \mi) --node[auto=false]{\ldots} (\ni x \mi);}{\draw (\k x \mi) -- (\ni x \mi);}
      ;

    \foreach \ni in {0,...,\n}
    \foreach \mi in {1,...,\m} 
      \pgfmathsetmacro\k{int(\mi-1)}
      \ifthenelse{\mi=5}{\path (\ni x \k) --node[above=-0.27cm,auto=false]{\vdots} (\ni x \mi);}{\draw (\ni x \k) -- (\ni x \mi);}
      ;

    \draw ({\l/2}, {\l/2}) node {$S_2$};
    \draw ({\l/2}, {\m * \l - \l/2}) node {$S_1$};
    \draw ({\n * \l - 3 * \l/2}, {\l/2}) node {$S_3$};

    \draw ({\l * \n / 2}, {-\l/2}) node {$n$};
    \draw ({-\l/2}, {\l * \m / 2}) node {$m$};
\end{tikzpicture}
\caption{Splitting of Nesterov's ``bad'' function in the main case $m,n \geq 3$, namely $m=20, n=80$. The vertical dimension corresponds to the inner graph $G'$, and the horizontal dimension~--- to the outer graph $G$}
\label{fig:func-splitting}
\end{figure}

\subsection{Proof of Theorem~\ref{thm:lower_time-varying}}
\begin{proof}
As in the proof of Theorem~\ref{thm:lower_static} we set the affine constraint in problem \ref{eq:lifted} to be $A_i x_i = 0$, with $A_i=A=\sqrt{W' \otimes I_{d/m}}$, where $W'$ is a gossip matrix of some inner communication graph $G'$.
Let the sequence of outer communication graphs $G(k)$ be the same as in the proof of Theorem 1 in \cite{kovalev2021lower}: $n= 3 \floor{\chi_W/3}$ nodes are split into three disjoint sets $V_1, V_2, V_3$ of equal size, and $G(k) = (V, E(k))$ are star graphs with the center nodes cycling through $V_2$.
Choose the inner communication graph $G' = (V', E')$ as in the proof of Theorem~\ref{thm:lower_static}.
Use Nesterov's function splitting given by \eqref{eq:splitting}, choose $S_{1G'}$ and $S_{2G'}$ as in the proof of Theorem~\ref{thm:lower_static}.
Set $S_1 = V_1 \times S_{1G'}$, $S_2 = V_1 \times S_{2G'}$ and $S_3 = V_3 \times E'$.
Setting $W(k)$ to be the Laplacian of the star graph $G(k)$ we have $\frac{\lmax(W(k))}{\lmin(W(k))} =n \leq \chi_W$.
Also (Lemma~2 \cite{kovalev2021lower} and proof of Theorem~\ref{thm:lower_static}), increasing the number of nonzero components of $x_k$  on any subnode requires local computation on a node in $S_1$, $\Theta\cbraces{\sqrt{\chi_A}}$ communications in the inner graph $G'$ (i.e. multiplications by $A^\top A$), one local computation on a node in $S_2$, $\Theta\cbraces{\chi_W}$ communications in the outer graph $G$ and one local computation on a node in $S_3$.
Same reasoning as in the proof of the previous theorem gives $\kappa_g = \Theta\cbraces{\frac{L_F}{\mu_F}}$, then using \eqref{eq:nesterov-lower} we conclude the proof.

\end{proof}

\subsection{Auxiliary lemmas for Theorem~\ref{thm:main}}
\begin{lemma}\label{dual:lem:descent}
	For $\theta \leq \frac{1}{\LH\lmax}$ we have the inequality
	\begin{equation}\label{dual:eq:1}
	\squeeze 
	H(\bz_f^{k+1}) \leq H(\bz_g^k) - \frac{\theta\lminp}{2}\sqn{\g H(\bz_g^k)}_{\mP}.
	\end{equation}
\end{lemma}

\begin{proof}
	We start with $\LH$-smoothness of $H$ on $\image \mP\mB$:
	\begin{align*}
	H(\bz_f^{k+1}) \leq H(\bz_g^k) + \<\g H(\bz_g^k), \bz_f^{k+1} - \bz_g^k> + \frac{\LH}{2}\sqn{\bz_f^{k+1} - \bz_g^k}.
	\end{align*}
	Using line~\ref{adom:line:zf} of Algorithm~\ref{alg:adom_common_constraints} together with \eqref{eq:spectrum} we get
	\begin{align*}
	H(\bz_f^{k+1}) &\leq H(\bz_g^k) - \theta \sqn{\g H(\bz_g^k)}_{\mW(k)} + \frac{\LH\theta^2}{2}\sqn{\g H(\bz_g^k)}_{\mW^2(k)}
	\\&\leq
	H(\bz_g^k) - \frac{\theta\lminp}{2}\sqn{\g H(\bz_g^k)}_{\mP} - \frac{\theta}{2}\sqn{\g H(\bz_g^k)}_{\mW(k)} + \frac{\LH\theta^2\lmax}{2}\sqn{\g H(\bz_g^k)}_{\mW(k)}
	\\&=
	H(\bz_g^k) - \frac{\theta\lminp}{2}\sqn{\g H(\bz_g^k)}_{\mP} +\frac{\theta}{2}\left(\theta\LH\lmax- 1\right)\sqn{\g H(\bz_g^k)}_{\mW(k)}.
	\end{align*}
	Using condition $\theta \leq \frac{1}{\LH\lmax}$ we get
	\begin{align*}
	H(\bz_f^{k+1})
	&\leq
	H(\bz_g^k) - \frac{\theta\lminp}{2}\sqn{\g H(\bz_g^k)}_{\mP}.
	\end{align*}
\end{proof}

\begin{lemma}\label{dual:lem:error}
	For $\sigma \leq \frac{1}{\lmax}$ we have the inequality 
	\begin{equation}\label{dual:eq:2}
	\begin{split}
	&\squeeze \sqn{\bm^k}_{\mP}
	\leq
	\left(1 - \frac{\sigma\lminp}{4}\right)\frac{4}{\sigma\lminp}\sqn{\bm^k}_{\mP}
	\\&\squeeze -
	\frac{4}{\sigma\lminp}\sqn{\bm^{k+1}}_\mP
	+
	\frac{8\eta^2}{(\sigma\lminp)^2}\sqn{\g H(\bz_g^k)}_\mP.
	\end{split}
	\end{equation}
\end{lemma}

\begin{proof}
	Using $\mP = \mP^2$ and $\mP \mW(k) = \mW(k) \mP = \mW(k)$
 together with lines~\ref{adom:line:delta} and~\ref{adom:line:m} of Algorithm~\ref{alg:adom_common_constraints} we obtain
	\begin{align*}
		\sqn{\bm^{k+1}}_\mP
		&=
		\sqn{\bm^{k} - \eta \g H(\bz_g^k) - \Delta^k}_\mP
		\\&=
		\sqn{(\mP -\sigma\mW(k))(\bm^{k} - \eta \g H(\bz_g^k))}
		\\&=
		\sqn{\bm^{k} - \eta \g H(\bz_g^k)}_\mP
		-
		2\sigma\sqn{\bm^{k} - \eta \g H(\bz_g^k)}_{\mW(k)}
		+
		\sigma^2\sqn{\bm^{k} - \eta \g H(\bz_g^k)}_{\mW^2(k)}.
	\end{align*}
	Using \eqref{eq:spectrum} we obtain
	\begin{align*}
	\sqn{\bm^{k+1}}_\mP
	&\leq
	\sqn{\bm^{k} - \eta \g H(\bz_g^k)}_\mP
	-
	\sigma\lminp\sqn{\bm^{k} - \eta \g H(\bz_g^k)}_{\mP}
	\\&-
	\sigma\sqn{\bm^{k} - \eta \g H(\bz_g^k)}_{\mW(k)}
	+
	\sigma^2\lmax\sqn{\bm^{k} - \eta \g H(\bz_g^k)}_{\mW(k)}
	\\&=
	\sqn{\bm^{k} - \eta \g H(\bz_g^k)}_\mP
	-
	\sigma\lminp\sqn{\bm^{k} - \eta \g H(\bz_g^k)}_{\mP}
	\\&+
	\sigma(\sigma\lmax - 1)\sqn{\bm^{k} - \eta \g H(\bz_g^k)}_{\mW(k)}.
	\end{align*}
	Using condition $\sigma \leq\frac{1}{\lmax}$ we get
	\begin{align*}
	\sqn{\bm^{k+1}}_\mP
	&\leq
	(1-\sigma\lminp)\sqn{\bm^{k} - \eta \g H(\bz_g^k)}_\mP.
	\end{align*}
	Using Young's inequality we get
	\begin{align*}
	\sqn{\bm^{k+1}}_\mP
	&\leq
	(1-\sigma\lminp)\left(
		\left(1 + \frac{\sigma\lminp}{2(1-\sigma\lminp)}\right)\sqn{\bm^{k}}_\mP
		+
			\left(1 + \frac{2(1-\sigma\lminp)}{\sigma\lminp}\right)\sqn{\eta \g H(\bz_g^k)}_\mP
	\right)
	\\&=
	\left(1 - \frac{\sigma\lminp}{2}\right)\sqn{\bm^k}_{\mP}
	+
	\eta^2\frac{(1-\sigma\lminp)(2-\sigma\lminp)}{\sigma\lminp}\sqn{\g H(\bz_g^k)}_\mP
	\\&\leq
	\left(1 - \frac{\sigma\lminp}{2}\right)\sqn{\bm^k}_{\mP}
	+
	\frac{2\eta^2}{\sigma\lminp}\sqn{\g H(\bz_g^k)}_\mP.
	\end{align*}
	Rearranging concludes the proof.
\end{proof}

\begin{lemma}\label{dual:lem:main}
	Let
	\begin{equation}\label{dual:alpha}
	\alpha = \frac{\muH}{2},
	\end{equation}
	\begin{equation}
	\eta = \frac{2\lminp}{7\lmax\sqrt{\muH \LH}},\label{dual:eta}
	\end{equation}
	\begin{equation}
		\theta = \frac{1}{\LH\lmax},\label{dual:theta}
	\end{equation}
	\begin{equation}
	\sigma = \frac{1}{\lmax},\label{dual:sigma}
	\end{equation}
	\begin{equation}
	\tau = \frac{\lminp}{7\lmax}\sqrt{\frac{\muH}{\LH}}.\label{dual:tau}
	\end{equation}
Define the Lyapunov function
	\begin{equation}\label{dual:psi}
		\Psi^k \defeq \sqn{\hat{\bz}^k - \bz^*} + \frac{2\eta(1-\eta\alpha)}{\tau}(F^*(\bz_f^k) - F^*(\bz^*) )+6\sqn{\bm^k}_{\mP},
	\end{equation}
	where $\hat{\bz}^k$ is defined as
	\begin{equation}\label{dual:zhat}
		\hat{\bz}^k = \bz^k + \mP \bm^k.
	\end{equation}
	Then the following inequality holds:
	\begin{equation}\label{dual:eq:recurrence}
		\Psi^{k+1} \leq \left(1-\frac{\lminp}{7\lmax}\sqrt{\frac{\muH}{\LH}}\right)\Psi^k.
	\end{equation}
\end{lemma}

\begin{proof}
	Using \eqref{dual:zhat} together with lines~\ref{adom:line:m} and~\ref{adom:line:z} of Algorithm~\ref{alg:adom_common_constraints}, we get
	\begin{align*}
	\hat{\bz}^{k+1} &= \bz^{k+1} + \mP \bm^{k+1}\\
	&= \bz^k + \eta\alpha(\bz_g^k - \bz^k) + \Delta^k + \mP( \bm^k - \eta\g H(\bz_g^k) - \Delta^k)
	\\&=
	\bz^k + \mP \bm^k + \eta\alpha(\bz_g^k - \bz^k) - \eta\mP \g H(\bz_g^k) + \Delta^k - \mP\Delta^k.
	\end{align*}
	From line~\ref{adom:line:delta} of Algorithm~\ref{alg:adom_common_constraints} and $\mP \mW(k) = \mW(k)$ it follows that $\mP\Delta^k = \Delta^k$, which implies
	\begin{align*}
	\hat{\bz}^{k+1} &= 
	\bz^k + \mP \bm^k + \eta\alpha(\bz_g^k - \bz^k) - \eta\mP \g H(\bz_g^k)
	\\&=
	\hat{\bz}^k + \eta\alpha(\bz_g^k - \bz^k) - \eta\mP \g H(\bz_g^k).
	\end{align*}
	Hence,
	\begin{align*}
		\sqn{\hat{\bz}^{k+1} - \bz^*}
		&=
		\sqn{\hat{\bz}^k - \bz^* + \eta\alpha(\bz_g^k - \bz^k) - \eta \mP \g H(\bz_g^k)}
		\\&=
		\sqn{(1 - \eta\alpha)(\hat{\bz}^k - \bz^* )+ \eta\alpha(\bz_g^k + \mP \bm^k - \bz^*)}
		+
		\eta^2\sqn{\g H(\bz_g^k)}_\mP
		\\&-
		2\eta\<\mP\g H(\bz_g^k), \bz^k  + \mP \bm^k- \bz^* + \eta\alpha(\bz_g^k - \bz^k)>
		\\&\leq.
  \end{align*}
  Using inequality 
  $\sqn{a + b} \leq (1+\gamma) \sqn{a} + (1 + \frac1\gamma) \sqn{b},~\gamma > 0$
  with $\gamma = \frac{\eta \alpha}{1- \eta \alpha}$ we get
	\begin{align*}
		\sqn{\hat{\bz}^{k+1} - \bz^*}
		&=
		(1-\eta\alpha)\sqn{\hat{\bz}^k - \bz^*} + \eta\alpha\sqn{\bz_g^k + \mP \bm^k - \bz^*}
		+
		\eta^2\sqn{\g H(\bz_g^k)}_\mP
		\\&-
		2\eta\<\g H(\bz_g^k),\mP(\bz_g^k - \bz^*)>
		+
		2\eta(1-\eta\alpha)\<\g H(\bz_g^k), \mP(\bz_g^k - \bz^k)>
		-
		2\eta\<\mP\g H(\bz_g^k),\bm^k>
		\\&\leq
		(1-\eta\alpha)\sqn{\hat{\bz}^k - \bz^*} + 2\eta\alpha\sqn{\bz_g^k - \bz^*}
		+
		2\eta\alpha\sqn{\bm^k}_\mP
		+
		\eta^2\sqn{\g H(\bz_g^k)}_\mP
		\\&-
		2\eta\<\g H(\bz_g^k),\mP(\bz_g^k - \bz^*)>
		+
		2\eta(1-\eta\alpha)\<\g H(\bz_g^k), \mP(\bz_g^k - \bz^k)>
		-
		2\eta\<\mP\g H(\bz_g^k),\bm^k>
	\end{align*}
	One can observe, that $\bz^k,\bz_g^k,\bz^* \in \image \mP \mB$. Hence,
	\begin{align*}
	\sqn{\hat{\bz}^{k+1} - \bz^*}
	&\leq
	(1-\eta\alpha)\sqn{\hat{\bz}^k - \bz^*} + 2\eta\alpha\sqn{\bz_g^k - \bz^*}
	+
	2\eta\alpha\sqn{\bm^k}_\mP
	+
	\eta^2\sqn{\g H(\bz_g^k)}_\mP
	\\&-
	2\eta\<\g H(\bz_g^k),\bz_g^k - \bz^*>
	+
	2\eta(1-\eta\alpha)\<\g H(\bz_g^k), \bz_g^k - \bz^k>
	-
	2\eta\<\mP\g H(\bz_g^k),\bm^k>.
	\end{align*}
	Using line~\ref{adom:line:zg} of Algorithm~\ref{alg:adom_common_constraints} we get
	\begin{align*}
	\sqn{\hat{\bz}^{k+1} - \bz^*}
	&\leq
	(1-\eta\alpha)\sqn{\hat{\bz}^k - \bz^*} + 2\eta\alpha\sqn{\bz_g^k - \bz^*}
	+
	2\eta\alpha\sqn{\bm^k}_\mP
	+
	\eta^2\sqn{\g H(\bz_g^k)}_\mP
	\\&-
	2\eta\<\g H(\bz_g^k),\bz_g^k - \bz^*>
	+
	2\eta(1-\eta\alpha)\frac{(1-\tau)}{\tau}\<\g H(\bz_g^k), \bz_f^k - \bz_g^k>
	-
	2\eta\<\mP\g H(\bz_g^k),\bm^k>.
	\end{align*}
	Using convexity and $\muH$-strong convexity of $H(\bz)$ on $\image \mP \mB$ we get
	\begin{align*}
	\sqn{\hat{\bz}^{k+1} - \bz^*}
	&\leq
	(1-\eta\alpha)\sqn{\hat{\bz}^k - \bz^*} + 2\eta\alpha\sqn{\bz_g^k - \bz^*}
	+
	2\eta\alpha\sqn{\bm^k}_\mP
	+
	\eta^2\sqn{\g H(\bz_g^k)}_\mP
	\\&-
	2\eta(H(\bz_g^k) - H(\bz^*)) - \eta\muH\sqn{\bz_g^k - \bz^*}
	\\&+
	2\eta(1-\eta\alpha)\frac{(1-\tau)}{\tau}(H(\bz_f^k) - H(\bz_g^k))
	-
	2\eta\<\mP\g H(\bz_g^k),\bm^k>
	\\&=
	(1-\eta\alpha)\sqn{\hat{\bz}^k - \bz^*}
	+
	\left(2\eta\alpha - \eta\muH\right)\sqn{\bz_g^k - \bz^*}
	+
	\eta^2\sqn{\g H(\bz_g^k)}_\mP
	\\&-
	2\eta(H(\bz_g^k) - H(\bz^*))
	+
	2\eta(1-\eta\alpha)\frac{(1-\tau)}{\tau}(H(\bz_f^k) - H(\bz_g^k))
	\\&-
	2\eta\<\mP\g H(\bz_g^k),\bm^k>
	+
	2\eta\alpha\sqn{\bm^k}_\mP.
	\end{align*}
	Using $\alpha$ defined by \eqref{dual:alpha} we get
	\begin{align*}
	\sqn{\hat{\bz}^{k+1} - \bz^*}
	&\leq
	\left(1-\frac{\eta \muH}{2}\right)\sqn{\hat{\bz}^k - \bz^*}
	+
	\eta^2\sqn{\g H(\bz_g^k)}_\mP
	\\&-
	2\eta(H(\bz_g^k) - H(\bz^*))
	+
	2\eta(1-\eta\alpha)\frac{(1-\tau)}{\tau}(H(\bz_f^k) - H(\bz_g^k))
	\\&-
	2\eta\<\mP\g H(\bz_g^k),\bm^k>
	+
	2\eta\alpha\sqn{\bm^k}_\mP.
	\end{align*}
	Since $H(\bz_g^k) \geq H(\bz^*)$, we get
	\begin{align*}
	\sqn{\hat{\bz}^{k+1} - \bz^*}
	&\leq
	\left(1-\frac{\eta\muH}{2}\right)\sqn{\hat{\bz}^k - \bz^*}
	+
	\eta^2\sqn{\g H(\bz_g^k)}_\mP
	\\&-
	2\eta(1-\eta\alpha)(H(\bz_g^k) - H(\bz^*))
	+
	2\eta(1-\eta\alpha)\frac{(1-\tau)}{\tau}(H(\bz_f^k) - H(\bz_g^k))
	\\&-
	2\eta\<\mP\g H(\bz_g^k),\bm^k>
	+
	2\eta\alpha\sqn{\bm^k}_\mP
	\\&=
	\left(1-\frac{\eta\muH}{2}\right)\sqn{\hat{\bz}^k - \bz^*}
	+
	\eta^2\sqn{\g H(\bz_g^k)}_\mP
	\\&+
	2\eta(1-\eta\alpha)\left(
	\frac{(1-\tau)}{\tau}H(\bz_f^k) + H(\bz^*) - \frac{1}{\tau}H(\bz_g^k)
	\right)
	\\&-
	2\eta\<\mP\g H(\bz_g^k),\bm^k>
	+
	2\eta\alpha\sqn{\bm^k}_\mP.
	\end{align*}
	Using \eqref{dual:eq:1} and $\theta$ defined by \eqref{dual:theta} we get
	\begin{align*}
	\sqn{\hat{\bz}^{k+1} - \bz^*}
	&\leq
	\left(1-\frac{\eta\muH}{2}\right)\sqn{\hat{\bz}^k - \bz^*}
	+
	\left(\eta^2-\frac{(1-\eta\alpha)\eta\lminp}{\tau\lmax\LH}\right)\sqn{\g H(\bz_g^k)}_\mP
	\\&+
	(1-\tau)\frac{2\eta(1-\eta\alpha)}{\tau}(H(\bz_f^k) - H(\bz^*) )
	-
	\frac{2\eta(1-\eta\alpha)}{\tau}(H(\bz_f^{k+1}) - H(\bz^*) )
	\\&-
	2\eta\<\mP\g H(\bz_g^k),\bm^k>
	+
	2\eta\alpha\sqn{\bm^k}_\mP.
	\end{align*}
	Using Young's inequality we get
	\begin{align*}
	\sqn{\hat{\bz}^{k+1} - \bz^*}
	&\leq
	\left(1-\frac{\eta\muH}{2}\right)\sqn{\hat{\bz}^k - \bz^*}
	+
	\left(\eta^2-\frac{(1-\eta\alpha)\eta\lminp}{\tau\lmax\LH}\right)\sqn{\g H(\bz_g^k)}_\mP
	\\&+
	(1-\tau)\frac{2\eta(1-\eta\alpha)}{\tau}(H(\bz_f^k) - H(\bz^*) )
	-
	\frac{2\eta(1-\eta\alpha)}{\tau}(H(\bz_f^{k+1}) - H(\bz^*) )
	\\&+
	\frac{\eta^2\lmax}{\lminp}\sqn{\g H(\bz_g^k)}_\mP + \frac{\lminp}{\lmax}\sqn{\bm^k}_\mP
	+
	2\eta\alpha\sqn{\bm^k}_\mP
	\\&=
	\left(1-\frac{\eta \muH}{2}\right)\sqn{\hat{\bz}^k - \bz^*}
	+
	\left(\eta^2+\frac{\eta^2\lmax}{\lminp}-\frac{(1-\eta\alpha)\eta\lminp}{\tau\lmax\LH}\right)\sqn{\g H(\bz_g^k)}_\mP
	\\&+
	(1-\tau)\frac{2\eta(1-\eta\alpha)}{\tau}(H(\bz_f^k) - H(\bz^*) )
	-
	\frac{2\eta(1-\eta\alpha)}{\tau}(H(\bz_f^{k+1}) - H(\bz^*) )
	\\&+
	\left(\frac{\lminp}{\lmax}+2\eta\alpha\right)\sqn{\bm^k}_\mP.
	\end{align*}
	Using \eqref{dual:eta} and \eqref{dual:alpha}, that imply $\eta\alpha\leq \frac{\lminp}{4\lmax}$, we obtain
	\begin{align*}
	\sqn{\hat{\bz}^{k+1} - \bz^*}
	&\leq
	\left(1-\frac{\eta \muH}{2}\right)\sqn{\hat{\bz}^k - \bz^*}
	+
	\left(\eta^2+\frac{\eta^2\lmax}{\lminp}-\frac{3\eta\lminp}{4\tau\lmax\LH}\right)\sqn{\g H(\bz_g^k)}_\mP
	\\&+
	(1-\tau)\frac{2\eta(1-\eta\alpha)}{\tau}(H(\bz_f^k) - H(\bz^*) )
	-
	\frac{2\eta(1-\eta\alpha)}{\tau}(H(\bz_f^{k+1}) - H(\bz^*) )
	\\&+
	\frac{3\lminp}{2\lmax}\sqn{\bm^k}_\mP.
	\end{align*}
	Using \eqref{dual:eq:2} and $\sigma$ defined by \eqref{dual:sigma} we get
	\begin{align*}
	\sqn{\hat{\bz}^{k+1} - \bz^*}
	&\leq
	\left(1-\frac{\eta\muH}{2}\right)\sqn{\hat{\bz}^k - \bz^*}
	+
	\left(\eta^2+\frac{\eta^2\lmax}{\lminp}-\frac{3\eta\lminp}{4\tau\lmax\LH}\right)\sqn{\g H(\bz_g^k)}_\mP
	\\&+
	(1-\tau)\frac{2\eta(1-\eta\alpha)}{\tau}(H(\bz_f^k) - H(\bz^*) )
	-
	\frac{2\eta(1-\eta\alpha)}{\tau}(H(\bz_f^{k+1}) - H(\bz^*) )
	\\&+
	\left(1 - \frac{\lminp}{4\lmax}\right)6\sqn{\bm^k}_{\mP}
	-
	6\sqn{\bm^{k+1}}_\mP
	+
	\frac{12\eta^2\lmax}{\lminp}\sqn{\g H(\bz_g^k)}_\mP
	\\&\leq
	\left(1-\frac{\eta\muH}{2}\right)\sqn{\hat{\bz}^k - \bz^*}
	+
	\left(\frac{14\eta^2\lmax}{\lminp}-\frac{3\eta\lminp}{4\tau\lmax\LH}\right)\sqn{\g H(\bz_g^k)}_\mP
	\\&+
	(1-\tau)\frac{2\eta(1-\eta\alpha)}{\tau}(H(\bz_f^k) - H(\bz^*) )
	-
	\frac{2\eta(1-\eta\alpha)}{\tau}(H(\bz_f^{k+1}) - H(\bz^*) )
	\\&+
	\left(1 - \frac{\lminp}{4\lmax}\right)6\sqn{\bm^k}_{\mP}
	-
	6\sqn{\bm^{k+1}}_\mP.
	\end{align*}
	Using $\eta$ defined by \eqref{dual:eta} and $\tau$ defined by \eqref{dual:tau} we get
	\begin{align*}
	\sqn{\hat{\bz}^{k+1} - \bz^*}
	&\leq
	\left(1-\frac{\lminp}{7\lmax}\sqrt{\frac{\muH}{\LH}}\right)\sqn{\hat{\bz}^k - \bz^*}
	+
	\left(1 - \frac{\lminp}{4\lmax}\right)6\sqn{\bm^k}_{\mP}
	-
	6\sqn{\bm^{k+1}}_\mP
	\\&+
	\left(1-\frac{\lminp}{7\lmax}\sqrt{\frac{\muH}{\LH}}\right)\frac{2\eta(1-\eta\alpha)}{\tau}(H(\bz_f^k) - H(\bz^*) )
	-
	\frac{2\eta(1-\eta\alpha)}{\tau}(H(\bz_f^{k+1}) - H(\bz^*) )
	\\&\leq
	\left(1-\frac{\lminp}{7\lmax}\sqrt{\frac{\muH}{\LH}}\right)
	\left(\sqn{\hat{\bz}^k - \bz^*} + \frac{2\eta(1-\eta\alpha)}{\tau}(H(\bz_f^k) - H(\bz^*) )+6\sqn{\bm^k}_{\mP}\right)
	\\&-
	\frac{2\eta(1-\eta\alpha)}{\tau}(H(\bz_f^{k+1}) - H(\bz^*) )
	-
	6\sqn{\bm^{k+1}}_\mP.
	\end{align*}
	Rearranging and using \eqref{dual:psi} concludes the proof.
\end{proof}

\subsection{Proof of Theorem~\ref{thm:main}}

\begin{proof}
    From derivation of the reformulated problem and Demyanov-Danskin theorem it follows that $\g F^*(\mB^\top \bz^*) = \bx^*$.
    Therefore $\g H(\bz^*) = (0_d, \bx^*)^\top$.
	Using $\LH$-smoothness of $H$ on $\image \mP \mB$ we get
	\begin{align*}
		\sqn{\g F^*(\mB^\top \bz_g^k) - \bx^*}
        &=
		\sqn{\g F^*(\mB^\top \bz_g^k) - F^*(\mB^\top \bz^*)}
		\leq
		\sqn{\g H(\bz_g^k) - \g H(\bz^*)}
		\leq
		\LH^2 \sqn{\bz_g^k - \bz^*}.
	\end{align*}
	Using line~\ref{adom:line:zg} of Algorithm~\ref{alg:adom_common_constraints} 
  and inequality 
  $\sqn{a + b} \leq (1+\gamma) \sqn{a} + (1 + \frac1\gamma) \sqn{b},~\gamma > 0$
  with $\gamma = \frac{1}{\tau} - 1$ we get
 we get
	\begin{align*}
		\sqn{\g F^*(\mB^\top \bz_g^k) - \bx^*}
		&\leq\tau\LH^2\sqn{\bz^k - \bz^*} + (1-\tau)\LH^2\sqn{\bz_f^k - \bz^*}.
	\end{align*}
	Using $\muH$-strong convexity of $H$ on $\image \mP \mB$ we get
	\begin{align*}
    \sqn{\g F^*(\mB^\top \bz_g^k) - \bx^*}
	&\leq\tau\LH^2\sqn{\bz^k - \bz^*} 
    +
    \frac{2(1-\tau) \LH^2}{\muH}(H(\bz_f^k) - H(\bz^*)).
	\end{align*}
	Using \eqref{dual:zhat} we get
	\begin{align*}
    &\sqn{\g F^*(\mB^\top \bz_g^k) - \bx^*}
    \\&\leq
    2\tau\LH^2\sqn{\hat{\bz}^k - \bz^*}
    +
    2\tau\LH^2\sqn{\bm^k}_\mP
    +
    \frac{2(1-\tau) \LH^2}{\muH}(H(\bz_f^k) - H(\bz^*))
    \\&=
    2\tau\LH^2\sqn{\hat{\bz}^k - \bz^*}
    +
    \frac{\tau(1-\tau)\LH^2}{\eta(1-\eta\alpha)\muH}\frac{2\eta(1-\eta\alpha)}{\tau}(H(\bz_f^k) - H(\bz^*))
    +
    \frac{\tau\LH^2}{3}6\sqn{\bm^k}_\mP.
    \\&\leq
    \max\left\{2\tau\LH^2,\frac{\tau(1-\tau)\LH^2}{\eta(1-\eta\alpha)\muH}, \frac{\tau \LH^2}{3}\right\}
    \left(\sqn{\hat{\bz}^k - \bz^*} + \frac{2\eta(1-\eta\alpha)}{\tau}(F^*(\bz_f^k) - F^*(\bz^*) )+6\sqn{\bm^k}_{\mP}\right)
    \\&=
    \max\left\{2\tau\LH^2,\frac{\tau(1-\tau)\LH^2}{\eta(1-\eta\alpha)\muH}\right\}
    \left(\sqn{\hat{\bz}^k - \bz^*} + \frac{2\eta(1-\eta\alpha)}{\tau}(F^*(\bz_f^k) - F^*(\bz^*) )+6\sqn{\bm^k}_{\mP}\right).
	\end{align*}
	Using the definition of $\Psi^k$ \eqref{dual:psi} and denoting 
    $C = \Psi^0	\max\left\{2\tau\LH^2,\frac{\tau(1-\tau)\LH^2}{\eta(1-\eta\alpha)\muH} \right\}$
    we get
	\begin{align*}
        \sqn{\g F^*(\mB^\top \bz_g^k) - \bx^*}
		 \leq \frac{C}{\Psi^0}\Psi^k.
	\end{align*}
	   Applying Lemma~\ref{dual:lem:main} concludes the proof. 
\end{proof}

\bibliography{references}

\end{document}